\documentclass{amsart}

\usepackage{amssymb} \usepackage{amsfonts} \usepackage{amsmath}
\usepackage{amsthm} \usepackage{epsfig} 
\usepackage{color}
\usepackage{pinlabel}
\usepackage{amscd}
\usepackage[all]{xy}
\usepackage{pinlabel}
\usepackage{tcolorbox}

\newtheorem{lemma}{Lemma}
\newtheorem{thm}[lemma]{Theorem}
\newtheorem{prop}[lemma]{Proposition}
\newtheorem{cor}[lemma]{Corollary}

\theoremstyle{definition}

\newtheorem{example}[lemma]{Example}

\theoremstyle{remark}
\newtheorem{rem}[lemma]{Remark}

\newcommand{\matR} {\ensuremath {\mathbb{R}}}
\newcommand{\matQ} {\ensuremath {\mathbb{Q}}}
\newcommand{\matZ} {\ensuremath {\mathbb{Z}}}

\newcommand{\matH} {\ensuremath {\mathbb{H}}}

\newcommand{\SO} {\ensuremath {{\rm SO}}}
\newcommand{\GL} {\ensuremath {{\rm GL}}}

\newcommand{\Or} {\ensuremath {{\rm O}}}

\author{Bruno Martelli}
\address{\newline
Dipartimento di Matematica,\newline
Largo Pontecorvo 5,\newline
56127 Pisa, Italy.}
\email{bruno.martelli@unipi.it}

\author{Alan W. Reid}
\address{\newline
Department of Mathematics,\newline
Rice University,\newline
Houston, TX 77005, USA.}
\email{alan.reid@rice.edu}

\title{The Dirac operator \\ on cusped hyperbolic manifolds}

\begin{document}

\begin{abstract}
We study how the spin structures on finite-volume hyperbolic $n$-manifolds restrict to cusps. When a cusp cross-section is a $(n-1)$-torus, there are essentially two possible behaviours: the spin structure is either \emph{bounding} or \emph{Lie}. We show that in every dimension $n$ there are examples where at least one cusp is Lie, and in every dimension $n\leq 8$ there are examples where all the cusps are bounding.  

By work of C. B\"ar, this implies that the spectrum of the Dirac operator is $\mathbb R$ in the first case, and discrete in the second.
We therefore deduce that there are cusped hyperbolic manifolds whose spectrum of the Dirac operator is $\mathbb R$ in all dimensions, and whose spectrum is discrete in all dimensions $n\leq 8$. 
\end{abstract}

\maketitle

\section*{Introduction} \label{introduction:section}

Let $M$ be a finite-volume hyperbolic $n$-manifold all of whose cusp cross-sections are Euclidean $(n-1)$-tori.
If $M$ has a spin structure, it induces one on every orientable codimension-1 submanifold, and hence in particular on every $(n-1)$-torus cusp cross-section. Recall that there are two spin structures on a $(n-1)$-torus $T$ up to automorphisms of $T$ (see \cite{B} for example), namely:
\begin{itemize}
\item the \emph{bounding} spin structure, induced by the representation of $T$ as the boundary of $D^2 \times S^1 \times \cdots \times S^1$ equipped with any spin structure,
\item the \emph{Lie} spin structure, induced by the representation of $T$ as a Lie group.
\end{itemize}


These spin structures are sometimes called \emph{non-trivial} and \emph{trivial}, respectively.
Our interest in the possible restrictions of spin structures to cusp cross-sections is motivated by the following theorem of B\"ar \cite{B}:

\begin{thm}[B\"ar] \label{Bar:thm}
Let $M$ be a finite volume hyperbolic $n$-manifold equipped with a spin structure all of whose cusp cross-sections are Euclidean $(n-1)$-tori. Then the spectrum of the Dirac operator of $M$ is:
\begin{enumerate}
\item $\matR$ if at least one cusp cross-section inherits the Lie spin structure;
\item discrete if all the cusp cross-sections inherit the bounding spin structure.
\end{enumerate}
\end{thm}

We will not define the Dirac operator here and refer the reader to \cite{B}. By explicit constructions, it is shown in \cite{B}  that both cases of Theorem \ref{Bar:thm} occur in dimensions $n=2,3$. Although it is known in dimensions $n\geq 4$,  that there are many cusped hyperbolic $n$-manifolds admitting a spin structure (indeed every finite-volume hyperbolic manifold is virtually spinnable \cite{LR, Su}), we are not aware of any example where the spin structure is constructed explicitly and its restriction to the cusp cross-sections determined. In particular it seems unknown whether cases (1) or (2) of Theorem \ref{Bar:thm} occur in dimension $n\geq 4$. The main purpose of this article is to provide a partial answer.

\begin{thm} \label{main:thm}
The following hold:
\begin{enumerate}
\item
For every integer $n\geq 2$ there is a finite volume cusped orientable hyperbolic $n$-manifold equipped with a spin structure,  all of whose cusp cross-sections are Euclidean $(n-1)$-tori, and 
such that the spin structure restricts to the Lie spin structure on at least one cusp cross-section;
\item for every integer $n\leq 8$ there is a finite volume cusped orientable hyperbolic $n$-manifold equipped with a spin structure, all of whose cusp cross-sections are Euclidean $(n-1)$-tori, and
such that the spin structure restricts to the bounding structure on all cusp cross-sections.
\end{enumerate}
\end{thm}

Combining this result with B\"ar's Theorem we get:

\begin{cor}
\label{cor:dirac}
The following hold:
\begin{enumerate}
\item
For every $n$ there is a finite volume cusped orientable hyperbolic $n$-manifold with a spin structure for which the spectrum of the Dirac operator is $\matR$;
\item 
for every $n\leq 8$ there is a finite volume cusped orientable hyperbolic $n$-manifold with a spin structure for which the spectrum of the Dirac operator is discrete.
\end{enumerate}
\end{cor}

The constructions of examples in (1) and (2) of Theorem \ref{main:thm} are very different in nature, and neither is completely straightforward, relying on results that have been proved recently. 

\subsection*{Acknowledgments} We thank Francesco Lin for introducing us to this problem. This work was initiated when the second author was visiting the Max-Planck-Institut f\"ur Mathematik, and he gratefully thanks them for their support and hospitality.

\section{The Lie group case}
\label{Lie}

\subsection{Preliminaries}
\label{Prelim} 
It will be convenient for this part of the discussion to identify hyperbolic $n$-space $\matH^n$, with the hyperboloid model, defined using the quadratic form $j_n:=x_0^2+x_1^2+\ldots x_{n-1}^2-x_n^2$; i.e.
$${\mathbb H}^n = \{x=(x_0,x_1,\ldots, x_n) \in {\mathbb R}^{n+1} : j_n(x)= -1, x_{n}>0\}$$
equipped with the Riemannian metric induced from the Lorentzian inner product associated to $j_n$. The full group of isometries of ${\mathbb H}^n$ is then identified with $\Or^+(n,1)$, the subgroup of
\[
\Or(n,1) = \{A \in \GL(n+1,\mathbb{R}) : A^tJ_nA = J_n\},
\]
preserving the upper sheet of the hyperboloid $j_n(x)=-1$, and
where $J_n$ is the symmetric matrix associated to the quadratic form $j_n$. The full group of orientation-preserving isometries is given by $\SO^+(n,1) = \{A\in \Or^+(n,1) : \det(A)= 1\}$.

If $M$ is a finite volume hyperbolic manifold admitting a spin structure (i.e. the first and second Stiefel Whitney classes are both zero), then the set of spin structures on $M$ can be identified with $H^1(M,\matZ/2\matZ)$. This identification is not canonical, although it can be made canonical after choosing a ``base" spin structure corresponding to $0\in H^1(M,\matZ/2\matZ)$.  Moreover, $H^1(M,\matZ/2\matZ)$ acts freely and transitively on the set of spin structures.

We will make use of the following lemma (c.f. \cite[Lemma 3]{B}).

\begin{lemma}
\label{get_trivial}
Let $M$ be a finite volume hyperbolic $n$-manifold admitting a spin structure, all of whose cusp cross-sections are Euclidean $(n-1)$-tori, and assume that for some cusp cross-section $T$, the inclusion map $T\hookrightarrow M$ induces a direct sum decomposition, 
$H_1(M,\matZ)=H_1(T,\matZ)\oplus A$ for some finitely generated Abelian group $A$. Then given any spin structure $\sigma$ on $T$, there exists a spin structure $s_\sigma$ on $M$ such that $s_\sigma$ restricts to $\sigma$ on $T$.
In particular this holds for the Lie spin structure on $T$.\end{lemma}

\begin{proof} Since $H_1(T,\matZ)$ is a direct summand of $H_1(M,\matZ)$, it follows that $H^1(M,\matZ/2\matZ)$ surjects on $H^1(T,\matZ/2\matZ)$. So given any spin structure $\sigma$ on $T$ we can find a spin structure $s_\sigma$ on $M$ such that $s_\sigma$ restricts to $\sigma$ on $T$.\end{proof}

\subsection{Cusped arithmetic hyperbolic manifolds}
\label{arithmetic}
The proof of Theorem \ref{main:thm}(1) will make use of cusped arithmetic, hyperbolic $n$-manifolds whose definition we now recall (see \cite{VS} for further details). 
Suppose that $X = \mathbb{H}^n/\Gamma$ is a finite volume cusped hyperbolic $n$-manifold. Then $X$ is arithmetic if $\Gamma$ is commensurable with a group $\Lambda < \SO^+(n,1)$ as described below.

Let $f$ be a non-degenerate quadratic form defined over $\matQ$ of signature $(n,1)$, which we can assume is diagonal and has integer coefficients. Then $f$ is equivalent over $\mathbb{R}$ to the form $j_n$ defined above; i.e. there exists $T\in \GL(n+1,\mathbb{R})$ such that $T^tFT = J_n$, where $F$ and $J_n$ denote
the symmetric matrices associated to $f$ and $j_n$ respectively. Then $T^{-1}\SO(f,\matZ)T\cap \SO^+(n,1)$ defines the arithmetic subgroup $\Lambda < \SO^+(n,1)$.

Note that the form $f$ is anisotropic (ie does not represent $0$ non-trivially over $\matQ$) if and only if the group $\Gamma$ is cocompact, otherwise the group $\Gamma$ is non-cocompact (see \cite{BHCh}).  By Meyer's Theorem  
\cite[\S IV.3.2, Corollary 2]{Se}, the case that $f$ is anisotropic can only occur when $n=2,3$.

The main result we will need from the theory of cusped arithmetic hyperbolic manifolds is that their fundamental groups are virtually special \cite{BHW}. We will not define this here, but we will prove a result which is a consequence of 
being virtually special.

\begin{thm}
\label{thm:BHW}
Let $\matH^n/\Lambda$ be a cusped arithmetic hyperbolic $n$-manifold.  Then $\Lambda$ has a finite index subgroup $\Gamma$ for which $\matH^n/\Gamma$ has all cusp cross-sections being Euclidean $(n-1)$-tori and for one such
cusp cross-section $T$, there exists a retraction $\Gamma\rightarrow \pi_1(T)$.\end{thm}

\begin{proof} We indicate how this follows from \cite{BHW}. By \cite{MRS} for example, we can first pass to a finite cover $M_1$ of $\matH^n/\Lambda$ so that all cusp cross-sections are Euclidean $(n-1)$-tori. Next, we use
\cite[Theorem 1.4]{BHW} together with the last sentence of \cite{BHW} which states: 

\smallskip

``{\em In fact, the proof of Theorem 1.4 extends to show that non-cocompact arithmetic lattices virtually retract onto their geometrically finite subgroups.}"

\smallskip

\noindent in the following way  (note that \cite[Theorem 1.4]{BHW} deals with certain closed arithmetic hyperbolic manifolds).

Let $T_1\subset M_1$ be a cusp cross-section, since $\pi_1(T_1)$ is geometrically finite, we can use the previous paragraph to arrange finite covers $M=\matH^n/\Gamma \rightarrow M_1$ and $T\rightarrow T_1$ (and $\pi_1(T)$ geometrically finite) together with a retraction from $\Gamma\rightarrow \pi_1(T)$. \end{proof}

\subsection{The proof of Theorem \ref{main:thm}(1)}
Let $X$ be a cusped arithmetic hyperbolic manifold of dimension $n$ (which we can assume is at least $4$ by the results of \cite{B}).  Using Sullivan's result \cite{Su} we can pass to a finite cover that admits a spin structure and also
(as noted in the proof of Theorem \ref{thm:BHW}), so that all cusp cross-sections are Euclidean $(n-1)$-tori. From Theorem \ref{thm:BHW},  we can pass to a further finite sheeted cover $M\rightarrow X$, that admits a spin structure
and for which $\pi_1(M)$ retracts onto $\pi_1(T)$.  This retraction determines a decomposition $H_1(M,\matZ)=H_1(T,\matZ)\oplus A$, for some finitely generated Abelian group $A$, and the result now 
follows from Lemma \ref{get_trivial}.\qed\\[\baselineskip]
We conclude \S \ref{Lie} with some remarks on arranging the Lie spin structure on multiple cusp cross-sections. As above, $M$ will be a cusped arithmetic hyperbolic manifold of dimension $n$ (of dimension at least $4$) admitting a spin structure with all cusp cross-sections $T_1, T_2, \ldots , T_s$ being Euclidean $(n-1)$-tori. For $d$ a positive integer, let $G_j^d$ be the characteristic subgroup of $\pi_1(T_j)$ arising as the kernel
of the homomorphism $\pi_1(T_j)\rightarrow (\matZ/d\matZ)^{n-1}$.  Standard combination techniques can be used to show that for $d$ sufficiently large, the group $G=G_1^d*G_2^d*\ldots *G_s^d$ is a geometrically finite subgroup of $\pi_1(M)$, and so we can apply \cite{BHW} (see above),
to obtain finite index subgroups $\Gamma < \pi_1(M)$ and $L<G$ with
$\Gamma$ admitting a retraction onto $L$. 

Note that since $[G:L]<\infty$, $L$ contains a finite index subgroup of each of $G_j^d$ for $j=1,\ldots s$, and each such will be isomorphic to $\matZ^{n-1}$ and conjugate into a 
peripheral subgroup of $\Gamma$. Furthermore, by the Kurosh subgroup theorem (see \cite[Chapter IV Theorem 1.10]{LS} for example), $L$ has the form  $F*(*_\alpha A_\alpha)$ where $F$ is a free group and $A_\alpha$ is the intersection of $L$ with a conjugate of some $G_j^d$. Since 
$[G:L]<\infty$ it follows that $A_\alpha\cong \matZ^{n-1}$ for each $\alpha$.  Putting this together we deduce that there is a retraction $\Gamma\rightarrow (\matZ^{n-1})^K$ where each $\matZ^{n-1}$ is a peripheral subgroup of $\Gamma$
and $K$ some large positive integer.

We can now use an extension of Lemma \ref{get_trivial} to get a direct sum decomposition,
$H_1(\Gamma,\matZ) =  (\matZ^{n-1})^K \oplus A$ for some finitely generated Abelian group $A$, and we can follow the proof of Lemma \ref{get_trivial} to obtain the Lie spin structure on a large number of cusp cross-sections of $\matH^n/\Gamma$. 

\section{The bounding case}
\label{bounding}
In this section we prove Theorem \ref{main:thm}(2), the method of proof being very different from that in \S \ref{Lie}. The core of the argument is a Dehn filling trick.

\subsection{The Dehn filling trick}
Let $M$ be a finite-volume hyperbolic $n$-manifold, all of whose cross-sections are $(n-1)$-tori. We say that a closed smooth $n$-manifold $N$ is a \emph{Dehn filling} of $M$ if $N$ contains some disjoint $(n-2)$-tori with trivial normal bundles whose complement is diffeomorphic to $M$. Given $M$, it is possible to construct a Dehn filling $N$ by attaching a copy of $D^2\times S^1 \times \cdots \times S^1$ to each truncated cusp along some diffeomorphism. The resulting $N$ of course depends on the chosen diffeomorphisms.

Here is a crucial observation: if a Dehn filling $N$ has a spin structure, it induces one in $M$ that is of bounding type on every cusp cross-section, because it extends by construction to the adjacent $D^2\times S^1 \times \cdots \times S^1$. The existence of a spin structure on $M$ that restricts to bounding spin structures on every cross-section is in fact equivalent to the existence of a spinnable Dehn filling $N$ of $M$. 

To prove Theorem \ref{main:thm}(2) it therefore suffices to exhibit a hyperbolic $n$-manifold that can be Dehn-filled to a spinnable closed $n$-manifold. We are able to do this for every $n\leq 8$ using right-angled polytopes and a vanishing theorem for Stiefel-Whitney classes of moment-angled manifolds recently discovered in \cite{HKK}.

\subsection{Right-angled polytopes} \label{right:subsection}
Let $P^3, \ldots, P^8$ be the notable sequence of finite-volume right-angled polytopes $P^n \subset \matH^n$ already considered by various authors \cite{ALR, ERT, IMM, PV}. These polytopes are combinatorially dual to the Euclidean \emph{Gosset polytopes} \cite{G} discovered by Gosset in 1900. The only information that we need here is that $P^n \subset \matH^n$ is right-angled and has at least one ideal vertex, its combinatorics will not be important.

Let $\Gamma$ be the reflection group of $P^n$. It follows easily from the standard Coxeter presentation of $\Gamma$ that its abelianisation is isomorphic to the finite group $(\matZ/2\matZ)^f$ where $f$ is the number of facets of $P^n$. It is also a standard fact that the kernel $\Gamma' = [\Gamma,\Gamma]$ of the abelianization contains no torsion and hence $M^n=\matH^n/\Gamma'$ is a cusped finite-volume hyperbolic manifold tessellated by $2^f$ copies of $P^n$. If we use the \emph{colouring} language as in \cite{IMM}, this is the manifold that we get by colouring all the facets of $P^n$ with distinct colours. It also follows from \cite[Proposition 7]{IMM} that all the cusps of $M^n$ are $(n-1)$-dimensional tori. The number of cusps may be very big.\footnote{The polytope $P^8$ has 2160 ideal vertices and 240 facets. The discussion in \cite[Section 1.2]{IMM} implies that $M^8$ has $2160 \cdot 2^{240-14} \sim 10^{71}$ cusps. This is still below the number of atoms in the observable universe, that is around $10^{80}$.} Note that the manifolds $M^3,\ldots, M^8$ defined here are much
larger than the manifolds considered in \cite{IMM}, that are some quotients of these. 

We will prove the following.

\begin{thm} \label{bounding:thm}
The hyperbolic cusped $n$-manifold $M^n$ has a spinnable Dehn filling. Therefore it has a spin structure where every cusp cross-section inherits a bounding spin structure. This holds for every $3\leq n \leq 8$.
\end{thm}

To prove this theorem we need to introduce some tools. We start by describing the Dehn fillings of $M^n$.

\subsection{Dehn fillings of polytopes with ideal vertices}
The following procedure works with every right-angled hyperbolic polytope containing some ideal vertex, but we focus on the poyhedra $P^3,\ldots, P^8$ for simplicity.

Consider $P^n$ as a Euclidean polytope, using the Klein model for $\matH^n$. We now \emph{Dehn fill} $P^n$ to produce a new abstract compact polytope $\bar P^n$, by substituting each ideal vertex of $P^n$ with a $(n-2)$-cube.

\begin{figure}
 \begin{center}
  \includegraphics[width = 12 cm]{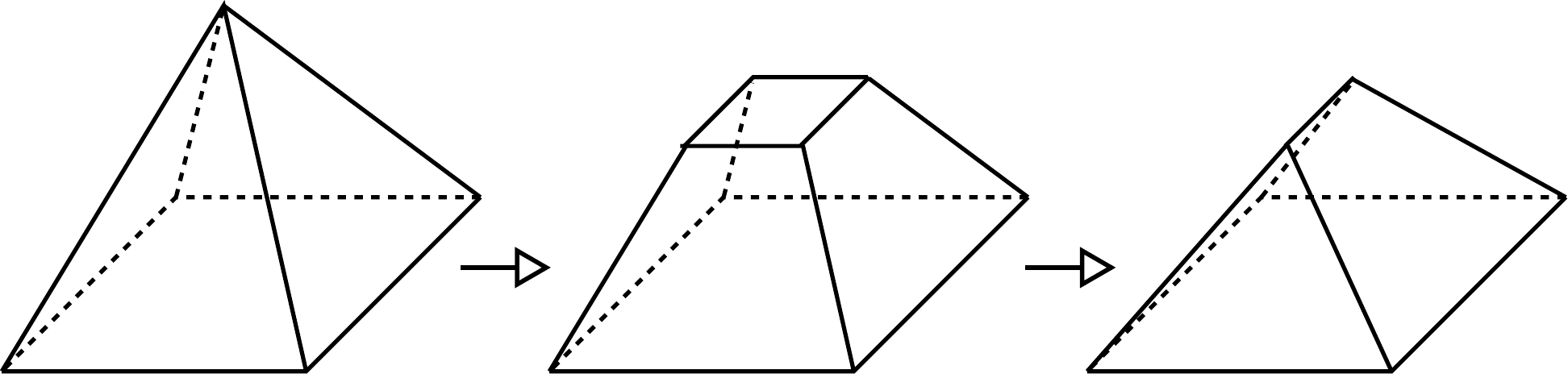}
 \end{center}
 \caption{We Dehn fill $P^n$ by replacing every ideal vertex with a $(n-2)$-cube. Here $n=3$. The resulting (abstract) polytope $\bar P^n$ is simple.}
 \label{Dehn_fill:fig}
\end{figure}

This operation goes as follows. At each ideal vertex $v$ of $P^n$, we first truncate it, thus producing a new small $(n-1)$-cubic facet $C$; then we choose two opposite facets $F_1,F_2$ of $C$, we foliate $C$ with $(n-2)$-cubes parallel to $F_1$ and $F_2$, and we identify all the leaves to a single $(n-2)$-cube $F$. See Figure \ref{Dehn_fill:fig}.

If we perform this operation at every ideal vertex $v$ of $P^n$ we get a topological disc $\bar P^n$ that has the structure of an \emph{abstract simple polytope}, that is its boundary is stratified into faces which intersect minimally, \emph{i.e.}~exactly $k$ of them at each $(n-k)$-stratum (this is a consequence of Remark \ref{Gosset:rem} below). Very often $\bar P^n$ may itself be realized as a polytope in $\matR^n$, but we will not need that. Note that $P^n$ is not simple precisely at the ideal vertices: therefore $\bar P^n$ may be seen as a perturbation of $P^n$ that transforms it into a simple polytope. 

The strata of $\bar P^n$ are the same of $P^n$, except that every ideal vertex $v$ is replaced with a $(n-2)$-cube $F$. The facets of $\bar P^n$ are those of $P^n$, with some additional adjacencies: at every ideal vertex $v$, the two facets of $P^n$ that were opposite with respect to $v$ and contained $F_1$ and $F_2$ are now adjacent in $\bar P^n$, since they intersect in the new $(n-2)$-cube $F$.

Note that at every ideal vertex $v$ there are $n-1$ different Dehn fillings to choose from, one for every pair of opposite facets $F_1,F_2$ in the small $(n-1)$-cube $C$.

\begin{figure}
 \begin{center}
  \includegraphics[width = 9 cm]{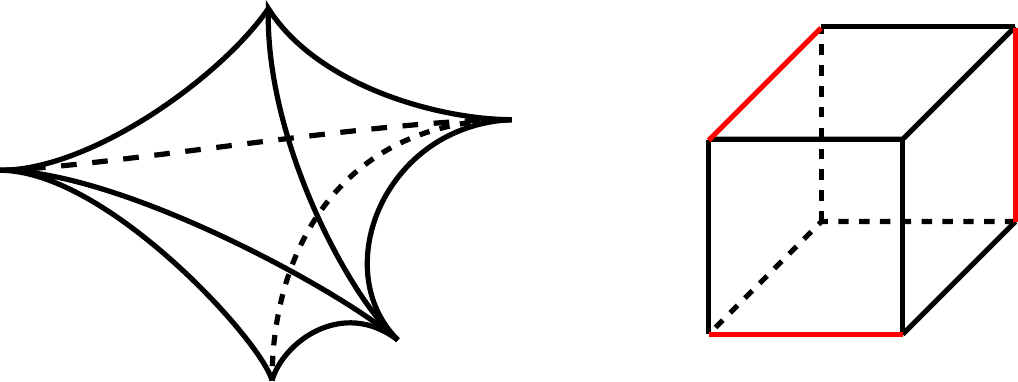}
 \end{center}
 \caption{The polytope $P^3$ is a right-angled bipyramid with three ideal vertices along the horizontal plane and two real ones (top and bottom in the figure). By Dehn filling $P^3$ we may obtain a cube: the three ideal vertices have been substituted with three edges (red in the figure).}
 \label{bipyramid:fig}
\end{figure}

\begin{example}
The polytope $P^3$ is the bipyramid shown in Figure \ref{bipyramid:fig}, with 3 ideal vertices and 2 real ones. 
The Dehn filling $\bar P^3$ is obtained by replacing every ideal vertex with an edge as in Figure \ref{Dehn_fill:fig}. At every ideal vertex there are two possible choices, so there are $2^3=8$ choices overall. In all cases we get an abstract simple polytope with 6 faces. It is possible to get a $\bar P^3$ that is combinatorially a cube as in Figure \ref{bipyramid:fig}.
\end{example}

\begin{rem} \label{Gosset:rem}
We may also describe $\bar P^n$ using dual polytopes. The Gosset polytope dual to $P^n$ is a Euclidean polytope $G^n \subset \matR^n$ with two types of facets: some simplexes, dual to the real vertices of $P^n$, and some cross-polytopes (also called \emph{hyper-octahedra}), dual to the ideal vertices of $P^n$. (Every ideal vertex of $P^n$ has a cubical link, hence the dual facet is a cross-polytope, that is dual to that cube.) 

Let $K^{n-1}$ be the simplicial complex obtained from the boundary of the Gosset polytope $G^n$ by subdividing each cross-polytope facet into $2^{n-1}$ simplexes as follows.
A cross-polytope in $\matR^{n-1}$ is the convex hull of $\pm e_1, \ldots \pm e_{n-1}$. Choose two opposite vertices, say $\pm e_1$ for simplicity, and subdivide the cross-polytope into the simplexes with vertices $e_1, -e_1, \pm e_2, \ldots, \pm e_n$. There are $2^{n-1}$ possible signs, hence $2^{n-1}$ simplexes. When $n=2,3$ this is the standard decomposition of a square into two triangles and of an octahedron into four simplexes. It depends on the choice of two opposite vertices, that is of a diagonal of the cross-polytope.

The stratification of $\partial \bar P^n$ is dual of the simplicial complex $K^{n-1}$. The choice of two opposite facets $F_1, F_2$ at each ideal vertex $v$ for $P^n$ corresponds to the choice of a diagonal in each cross-polytope of $G^n$, and the additional $(n-2)$-cube $F$ in $\bar P^n$ corresponds to the additional diagonal in $K^{n-1}$. Perturbing $P^n$ to the simple polyhedron $\bar P^n$ corresponds dually to subdividing the complex $\partial G^n$ into the simplicial complex $K^{n-1}$.
\end{rem}

We think of $\bar P^n$ as an abstract right-angled compact polytope, that is a topological disc with right-angled corners. In some fortunate cases as in Figure \ref{bipyramid:fig} the polyhedron $\bar P^n$ may indeed be interpreted as a right-angled polyhedron in some geometry ($\bar P^3$ is a Euclidean cube in the figure), but this does not hold in general, and we do not need it. 

We may assign distinct colours to all the facets of $\bar P^n$ and apply the colouring construction to $\bar P^n$ as in \cite{IMM}. The result is a closed topological manifold $\bar M^n$. 

The manifold $\bar M^n$ is naturally a Dehn filling of $M^n$. Using the techniques of \cite{IMM} we see that $\bar M^n$ decomposes into $2^f$ identical copies of $\bar P^n$, and that the pre-image of every additional $(n-2)$-cube $F$  of $\bar P^n$ in $\bar M^n$ consists of many $(n-2)$-tori, each tessellated into $2^{2n-4}$ copies of $F$. Since $\bar P^n$ is obtained from $P^n$ by substituting ideal vertices with $(n-2)$-cubes, the manifold $M^n$ is naturally homeomorphic to $\bar M^n$ minus all the $(n-2)$-tori that are the pre-images of these $(n-2)$-cubes $F$. The closed manifold $\bar M^n$ also inherits a smooth structure from $M^n$.

\begin{example}
We Dehn fill the polyhedron $P^3$ to a cube $\bar P^3$. The filled manifold $\bar M^3$ is a 3-torus tessellated into $2^6$ copies of the cube $\bar P^3$. Therefore $M^3$ is the complement of a 12-components link in the 3-torus $\bar M^3$. The link consists of the pre-images of the red edges in the cube shown in Figure \ref{bipyramid:fig}. In fact the containment of $M^3$ in the 3-cube is a 8-fold covering of the usual description of the Borromean ring complement as a link complement in the 3-torus. So in particular $M^3$ is an 8-fold cover of the Borromean ring complement, and in fact $P^3$ is commensurable with the ideal regular right-angled octahedron.
\end{example}

\subsection{Dehn fillings are spinnable}
To conclude the proof of Theorem \ref{bounding:thm}, and hence of Theorem \ref{main:thm}(2), it remains to show that the Dehn filling $\bar M^n$ constructed above is spinnable. This is an instance of a more general theorem proved recently by Hasui -- Kishimoto -- Kizu in \cite{HKK} in the context of \emph{moment angle manifolds}. This theorem shows in fact that all the Stiefel-Whitney classes of $\bar M^n$ vanish. 

We briefly recall this setting. Let $K$ be a simplicial complex with vertices $\{1,\ldots, m\}$. The \emph{real moment-angle complex} determined by $K$ is 
$$\matR Z_K = \bigcup_{\sigma \in K} \matR Z_\sigma \subset [-1,1]^m$$
where 
$$\matR Z_\sigma = X_1 \times \cdots \times X_m$$
such that $X_i$ equals $[-1,1]$ if $i\in\sigma$ and $\{-1,1\}$ if $i \not \in \sigma$. The symbol $\matR$ is used only to distinguish the real moment-angle complex from its complex version $Z_K$, that we will not use here. 

By construction $\matR Z_K$ is a cube subcomplex of $[-1,1]^m$ that is symmetric with respect to the $m$ reflections along the hyperplanes $x_i = 1/2$, $i=1,\ldots, m$. These symmetries act transitively on the $2^m$ vertices of the cube complex. It is quite easy to check that the link of a vertex in $\matR Z_K$ is isomorphic to the simplicial complex $K$ itself. In particular, if $K$ is homeomorphic to a $(n-1)$-sphere then $M=\matR Z_k$ is a topological $n$-manifold. The manifold $M$ is only topological, but it is noted in \cite{F, HKK} that Stiefel -- Whitney classes need no smooth structure to be defined, and moreover the following holds.

\begin{thm} [Hasui -- Kishimoto -- Kizu] \label{HKK:thm}
If $K$ is a topological sphere, the Stiefel-Whitney classes of the real moment-angle manifold $\matR Z_K$ all vanish.
\end{thm}

This theorem applies to our Dehn fillings $\bar M^n$ because of the following.

\begin{prop}
Let $K=K^{n-1}$ be the simplicial complex constructed by subdividing the Gosset polytope in Remark \ref{Gosset:rem}. The real moment-angle manifold $\matR Z_K$ is homeomorphic to $\bar M^n$, for every $3\leq n \leq 8$.
\end{prop}
\begin{proof}
Recall that $K^{n-1}$ is dual to $\partial \bar P^n$ and that $\bar M^n$ is obtained from $\bar P^n$ by colouring its facets with all distinct colours. The decomposition of $\bar M^n$ into identical copies of $\bar P^n$ is dual to a cube complex $C$ described in \cite[Section 2]{IMM}. The cube complexes $\matR Z_K$ and $C$ are in fact isomorphic by construction.
\end{proof}

By combining these two results we get that the Stiefel -- Whitney classes of $\bar M^n$ vanish, so in particular $\bar M^n$ is spinnable and Theorem \ref{main:thm}(2) is proved.\qed

\section{Final remarks and questions}
\label{final}
Motivated by B\"ar's result Theorem \ref{Bar:thm}, the most interesting question about the nature of the spectrum of the Dirac operator that remains after out work is captured by the following. \\[\baselineskip]
\noindent{\bf Question 1:}~{\em For $n\geq 9$, does there exist examples of cusped orientable hyperbolic $n$-manifolds that have all cusp cross-sections being $(n-1)$-tori, which admits a spin structure that restricts to each cusp cross-section as the bounding spin structure?}\\[\baselineskip]
\noindent As noted in \S \ref{bounding}, this is equivalent to the existence of a spinnable Dehn filling $N$ of the cusped hyperbolic manifold. Thus we pose:\\[\baselineskip]
\noindent{\bf Question 2:}~{\em For each $n\geq 9$, does there exist a cusped orientable hyperbolic $n$-manifold $M$ having all cusp cross-sections being $(n-1)$-tori which admits a Dehn filling $N$ (in the sense of \S \ref{bounding}) a closed $n$-manifold that is spinnable?}

\end{document}